\documentclass[11pt]{article}

\usepackage[margin=1in]{geometry}
\usepackage[T1]{fontenc}
\usepackage{enumitem}
\usepackage[utf8]{inputenc}
\usepackage{amsmath,amssymb,amsthm,amsfonts}
\usepackage{bbm}
\usepackage{xspace}
\usepackage{url}
\usepackage[
    CJKbookmarks=true,
    bookmarksnumbered=true,
    bookmarksopen=true,
    colorlinks=true,
    citecolor=red,
    linkcolor=blue,
    anchorcolor=red,
    urlcolor=blue
]{hyperref}
\usepackage{cleveref}

\newcommand{\bE}{\mathbb{E}}
\newcommand{\bP}{\mathbb{P}}
\newcommand{\bR}{\mathbb{R}}
\newcommand{\var}{\mathrm{Var}}
\newcommand{\cov}{\mathrm{Cov}}

\newcommand{\indc}[1]{{\mathbf{1}_{\left\{{#1}\right\}}}}
\newcommand{\pth}[1]{\left( #1 \right)}
\newcommand{\qth}[1]{\left[ #1 \right]}

\theoremstyle{plain}
\newtheorem{thm}{Theorem}[section]
\newtheorem{lemma}[thm]{Lemma}
\newtheorem{corollary}[thm]{Corollary}

\theoremstyle{definition}
\newtheorem{definition}[thm]{Definition}

\crefname{lemma}{Lemma}{Lemmas}
\Crefname{lemma}{Lemma}{Lemmas}
\crefname{corollary}{Corollary}{Corollaries}
\Crefname{corollary}{Corollary}{Corollaries}
\crefname{conjecture}{Conjecture}{Conjectures}
\Crefname{conjecture}{Conjecture}{Conjectures}
\crefname{definition}{Definition}{Definitions}
\Crefname{definition}{Definition}{Definitions}
\crefname{thm}{Theorem}{Theorems}
\Crefname{thm}{Theorem}{Theorems}
\crefformat{equation}{(#2#1#3)}

\title{A Sharp Refinement of Samuels' Inequality via Cap Moves}
\author{
    Yanjun Han\thanks{New York University, \url{yanjunhan@nyu.edu}}
    \and 
    Zipei Nie\thanks{University of Illinois Urbana-Champaign, \url{znie@illinois.edu}}  
}
\date{\today}

\begin{document}

\maketitle

\begin{abstract}
Let $0\le \mu_1\le \mu_2 \le \cdots \le \mu_n$ and $\delta > 0$. Samuels' conjecture states that if $X_1,\dots,X_n$ are independent nonnegative random variables with $\bE[X_i] = \mu_i$, then
\begin{align*}
	\bP\pth{ \sum_{i=1}^n X_i < \delta + \sum_{i=1}^n \mu_i } \ge \min_{1\le i\le n} \prod_{j=i}^n \pth{1-\frac{\mu_j}{\delta + \sum_{k=i}^n \mu_k}}. 
\end{align*}
This conjecture was recently proved by Ling \cite{Ling2026}. In this note, we prove a refinement of Samuels' inequality via an operation called the \emph{cap move}. Our refined inequality is sharp, with equality attained for a strictly larger class of distributions than in the original inequality.
\end{abstract}

\section{Introduction}

A classical problem in probability is to determine sharp lower bounds for the probability that a sum of independent nonnegative random variables falls below a given threshold, subject to constraints on their expectations. Samuels conjectured an exact bound for this problem \cite{Samuels1966,Samuels1968,Samuels1969}, which was recently proved by Ling \cite{Ling2026}.

\begin{thm}\label{conj:samuels}\cite{Ling2026}
Let $X_1,\dots,X_n$ be independent nonnegative random variables with $\mathbb E[X_i]=\mu_i$ and $0\le \mu_1\le \mu_2\le\cdots\le \mu_n.$
Then, for every $\delta>0$,
\[\mathbb P\left(\sum_{i=1}^n X_i<\sum_{i=1}^n \mu_i+\delta
\right)\ge\min_{1\le k\le n}\prod_{i=k}^n
\left(1-\frac{\mu_i}{\delta+\sum_{j=k}^n\mu_j}\right).
\]
\end{thm}

Theorem~\ref{conj:samuels} resolves Feige's conjecture \cite{feige2004,feige2006} and has applications to problems in extremal combinatorics \cite{alon2012large,ferber2019uniformity}. Although Theorem~\ref{conj:samuels} is sharp, we can obtain a stronger bound by incorporating additional information about a conditional expectation of the sum. The new lower bound coincides with the $k$-th expression in Theorem~\ref{conj:samuels} when the conditional expectation equals $\sum_{i=1}^{k-1}\mu_i$.

\begin{thm}\label{thm:refined-samuels} Let $X_1,\dots,X_n$ be independent nonnegative random variables with $\mathbb E[X_i]=\mu_i$ and $0\le \mu_1\le\cdots\le\mu_n$. Let $\delta$ be an arbitrary positive real number, and let $E$ denote the conditional expectation $$E:=\mathbb E\left( \sum_{i=1}^n X_i \,\middle|\, \sum_{i=1}^n X_i<\sum_{i=1}^n\mu_i+\delta \right)\le \sum_{i=1}^n\mu_i.$$
Then, we have \[ \mathbb P\left( \sum_{i=1}^n X_i<\sum_{i=1}^n\mu_i+\delta \right) \ge \left( 1- \frac{ \sum_{i=1}^{K}\mu_i- E }{ \delta+\sum_{i=1}^n\mu_i- E } \right) \prod_{i=K+1}^n \left( 1- \frac{\mu_i}{ \delta+\sum_{j=1}^n\mu_j- E } \right),\]
where $K$ is the smallest index such that \( E\le  \sum_{i=1}^K\mu_i.\)
\end{thm}

A distinctive feature of Theorem~\ref{thm:refined-samuels} is that the bound is sharp for every admissible value of the conditional expectation $E$. Moreover, equality holds for a broader class of distributions than in the original Samuels' inequality. The equality cases are characterized by the following theorem.

\begin{thm}\label{thm:equality}
Under the assumptions and notation of Theorem~\ref{thm:refined-samuels}, equality holds if and only if one of the following conditions is satisfied:
\begin{enumerate}[label=(\alph*)]
    \item
    \(
    \sum_{i=1}^n X_i
    \le  \sum_{i=1}^n\mu_i+\delta
    \) almost surely.
    \item There exist $\alpha_1,\dots,\alpha_n\ge 0$ such that
    \[
    X_i\in\left\{
    \alpha_i,\,
    \alpha_i+\delta+\sum_{j=1}^n\mu_j-E
    \right\}
    \]
    for each $1\le i\le n$ almost surely, and
    \[
    \{\mu_i-\alpha_i:1\le i\le n\}
    =
    \left\{
    0,\ldots,0,\,
    \sum_{i=1}^K\mu_i-E,\,
    \mu_{K+1},\ldots,\mu_n
    \right\}
    \]
    as multisets.
\end{enumerate}
\end{thm}

In \Cref{thm:equality}, condition~(a) is equivalent to $\bP(\sum_{i=1}^nX_i<\sum_{i=1}^n \mu_i+\delta)=\delta/(\delta+\sum_{i=1}^n \mu_i -E)$, that is, to equality in the simple lower bound from Markov's inequality: 
\begin{align*}
\bE\pth{\sum_{i=1}^n X_i} \ge E\cdot \bP\pth{\sum_{i=1}^n X_i < \sum_{i=1}^n \mu_i + \delta} + \pth{\sum_{i=1}^n \mu_i + \delta} \pth{1 - \bP\pth{\sum_{i=1}^n X_i < \sum_{i=1}^n \mu_i + \delta}}. 
\end{align*}
In condition~(b), after possibly permuting coordinates with equal means, the multiset identity is equivalent to $\alpha_i=0$ for $i>K$ and $\alpha_i=\mu_i$ for $i\le K$ except possibly one index $j\le K$, at which $\alpha_j=\mu_j-r$, with $r=\sum_{i=1}^K \mu_i - E$. This exceptional index can be chosen arbitrarily subject to $\mu_j\ge r$. Thus, among the first $K$ random variables, all but possibly one are deterministic, while every random variable with index $i>K$ is supported on $\{0,\delta+\sum_{k=1}^n \mu_k -E\}$.

When $\mu_i = 1$ for all $1\le i\le n$, Theorem~\ref{thm:refined-samuels} immediately implies the following sharp refinement of Feige's inequality. 

\begin{corollary}\label{thm:refined-feige} Let $X_1,\dots,X_n$ be independent nonnegative random variables with $\mathbb E[X_i]=1$. Let $\delta$ be an arbitrary positive real number, and let $E$ denote the conditional expectation \[ E:=\mathbb E\left( \sum_{i=1}^n X_i \,\middle|\, \sum_{i=1}^n X_i<n+\delta \right)\le n. \] Then, we have \[ \mathbb P\left( \sum_{i=1}^n X_i<n+\delta \right) \ge \left( 1-\frac{\left\lceil E\right\rceil-E}{n+\delta-E} \right) \left( 1-\frac{1}{n+\delta-E} \right)^{n-\left\lceil E\right\rceil}. \] \end{corollary}

Equality holds if and only if either $\sum_{i=1}^n X_i\le n+\delta$ almost surely, or there exist $\alpha_1,\dots,\alpha_n\in[0,1]$ with at most one lying in the interior $(0,1)$, such that $\sum_{i=1}^n\alpha_i=E$ and $X_i\in\{\alpha_i,\alpha_i+n+\delta-E\}$ almost surely.

In Section~\ref{sec:2}, we prove Theorem~\ref{thm:refined-samuels} and Theorem~\ref{thm:equality} using an operation called the \emph{cap move}. In Section~\ref{sec:3}, we show that Theorem~\ref{thm:refined-samuels} implies the original Samuels' inequality (Theorem~\ref{conj:samuels}).

\paragraph{AI disclosure.} The authors used GPT Sol 5.6 Pro and Opus 5.5 for mathematical discussions. The authors take full responsibility for the content of this note.

\section{Proofs of Theorem~\ref{thm:refined-samuels} and Theorem~\ref{thm:equality}}\label{sec:2}

Our proof relies on a recursive procedure that modifies the random variables $(X_1,\dots,X_n)$ over time. Before our formal proof, it is therefore useful to distinguish what changes from what does not, and to track several key quantities throughout the procedure.

The only quantity that remains fixed is $\lambda := \sum_{i=1}^n \mu_i + \delta$, the target upper bound. For random variables $(X_1,\dots,X_n)$, let $(\mu_1,\dots,\mu_n)$ be the mean vector, and $S=\sum_{i=1}^n X_i$ be the sum. Whenever $\bP(S<\lambda)>0$, let
\begin{align*}
E := \bE[S\mid S<\lambda]
\end{align*}
be the conditional expectation, which satisfies $0\le E\le \bE[S]=\sum_{i=1}^n \mu_i$. We denote their \emph{gap} by 
\begin{align*}
G := \sum_{i=1}^n \mu_i - E \ge 0. 
\end{align*}
When $(X_1,\dots,X_n)$ is changed into $(X_1',\dots,X_n')$, the quantities $(\mu_1,\dots,\mu_n,S,E,G)$ are changed into $(\mu_1',\dots,\mu_n',S',E',G')$. Our probability of interest is $\bP(S<\lambda)$. 

\subsection{Cap move}\label{subsec:cap_move}
Our basic procedure of modifying the random variables $(X_1,\dots,X_n)$ is called the \emph{cap move}, formally defined below. 

\begin{definition}[Cap move]\label{defn:cap}
A cap move selects an index $j\in[n]$ and a real number $a\ge 0$, and replaces $X_j$ with $X_j'=\min(X_j,a)$. The other random variables remain unchanged, so $X_i'=X_i$ for all $i\neq j$. 
\end{definition}

Under the notation at the beginning of \Cref{sec:2}, it is clear that $S'\le S$ after a cap move. Therefore, $\bP(S<\lambda)\le \bP(S'<\lambda)$. The next lemma shows an opposite inequality. 

\begin{lemma}\label{lem:one-move}
    Use the notation at the beginning of \Cref{sec:2}. If $\bP(S<\lambda)>0$, after each cap move, we have    $$\bP\left(S < \lambda \right)\ge \pth{1-\frac{G-G'}{\lambda - E}}\bP\left(S'<\lambda\right) .$$
In particular, we have $G'\le G$. 
\end{lemma}
\begin{proof}
Suppose the cap move replaces $X_j$ by $X_j'=\min(X_j,a)$, leaving the other variables unchanged. Set
\[
S_{-j}:=\sum_{i\ne j}X_i,\qquad D:=X_j-X_j'=(X_j-a)_+,\qquad d:=\bE[D]=\mu_j-\mu_j',
\]
where $x_+:=\max(x,0)$, and write
\[
p:=\bP(S<\lambda),\qquad p':=\bP(S'<\lambda).
\]
Since $\bP(S<\lambda)>0$, we have $p'\ge p > 0$. The definitions of $E$ and $E'$ then give
\[
\bE[(\lambda-S)_+]=(\lambda-E)p,\qquad \bE[(\lambda-S')_+]=(\lambda-E')p'.
\]

We claim that, pointwise,
\[
0\le (\lambda-S')_+-(\lambda-S)_+\le D\indc{S_{-j}+a<\lambda}.
\]
Indeed, if $X_j\le a$, then $D=0$ and $S'=S$. If $X_j>a$, then $S'=S_{-j}+a$ and $S=S'+D$; the difference vanishes when $S_{-j}+a\ge\lambda$ and is at most $D$ otherwise. Since the random variables $X_i$ are independent, $D$ is independent of $S_{-j}$. Moreover, $X_j'\le a$ implies $\{S_{-j}+a<\lambda\}\subseteq\{S'<\lambda\}$. Taking expectations therefore yields
\begin{align*}
(\lambda-E')p'-(\lambda-E)p\le \bE[D\indc{S_{-j}+a<\lambda}]=d\,\bP(S_{-j}+a<\lambda)\le dp'.
\end{align*}
Only the $j$-th mean changes, so the definitions of the gaps give
\[
G-G'=\pth{\sum_i\mu_i-E} - \pth{\sum_i\mu_i'-E'}=d+E'-E.
\]
Consequently,
\[
(\lambda-E)p\ge(\lambda-E'-d)p' =\pth{\lambda-E-(G-G')}p'.
\]
Dividing by $\lambda-E>0$ proves the stated probability inequality. Finally, since $p'\ge p>0$, the stated inequality implies $G'\le G$. 
\end{proof}

\Cref{lem:one-move} will serve as the central inequality for our recursion. Since $G'\le G$, a key quantity we will track is the decrement $G-G'$ in the numerator. To make sure the denominator $\lambda-E$ can also be controlled, we consider a special type of cap move called the \emph{safe move} in the next subsection. 

\subsection{Safe move and its existence}\label{subsec:safe_move}
\begin{definition}[Safe move]\label{defn:safe}
    A cap move is called \emph{safe} if $E'\le E$, and \emph{strictly safe} if $E'< E$.
\end{definition}

The main motivation behind the safe move is that, after a sequence of safe moves, we can lower bound the denominator $\lambda-E$ in \Cref{lem:one-move} by a common quantity $\lambda-E^{(0)}$, where $E^{(0)}$ is the value of $E$ at the very beginning. This will be used in the proof of \Cref{lem:telescoping}. 

A key result in this subsection is that a strictly safe move always exists if $\var(S\mid S<\lambda)>0$. 

\begin{lemma}\label{lem:safe}
If $\bP(S<\lambda)>0$ and $\var(S\mid S<\lambda)>0$, then there exists a strictly safe cap move.
\end{lemma}

The rest of this subsection is devoted to the proof of this result. The following identity is the key to the proof. 

\begin{lemma}\label{lem:covariance}
Suppose $\bP(S<\lambda)>0$. Let $E_{j,a}$ denote the updated conditional expectation $E'$ after performing the cap move with index $j\in[n]$ and $a\ge 0$. Then, there exists a nonnegative measure $\nu_j$ on $\bR_+$ such that
\begin{align*}
	\cov\left(S, X_j\, \middle |\, S < \lambda \right) = \int_{[0,\infty)} (E-E_{j,a})\,\nu_j(da). 
\end{align*}
\end{lemma}
\begin{proof}
Fix $j\in[n]$, and write $p:=\bP(S<\lambda)>0$ and $S_{-j}:=\sum_{i\ne j}X_i$. For $a\ge0$, set $S_{j,a}:=S_{-j}+\min(X_j,a)$ and $p_{j,a}:=\bP(S_{j,a}<\lambda)$. Since $S_{j,a}\le S$, we have $p_{j,a}\ge p>0$, so $E_{j,a}$ is well-defined. Let $P_j$ and $P_{j,a}$ be the distributions of $X_j$ and $\min(X_j,a)$, respectively, and define
\[
g_j(x):=\bE\qth{(E-S_{-j}-x)\indc{S_{-j}+x<\lambda}},\qquad x\ge0.
\]
Independence and the definitions of the conditional expectations give
\begin{align*}
\int g_j(x)\,P_j(dx)&=\bE\qth{(E-S)\indc{S<\lambda}}=0,\\
\int g_j(x)\,P_{j,a}(dx)&=\bE\qth{(E-S_{j,a})\indc{S_{j,a}<\lambda}}=(E-E_{j,a})p_{j,a}.
\end{align*}
Consequently,
\[
\cov(S,X_j\mid S<\lambda)=\bE[(S-E)X_j\mid S<\lambda]=-\frac{1}{p}\int xg_j(x)\,P_j(dx)=\frac{1}{p}\int(\mu_j-x)g_j(x)\,P_j(dx).
\]

We now construct a finite nonnegative measure $\nu_{j,0}$ supported on $[0,\lambda)$. Let $b_j:=\operatorname{ess\,sup}X_j\in[0,\infty]$, and define
\[
w_j(t):=\begin{cases}
\bE[X_j\mid X_j>t]-\mu_j,&0\le t<b_j,\\
b_j-\mu_j,&t\ge b_j.
\end{cases}
\]
Since $t\mapsto \bE[X_j\mid X_j>t]$ is nondecreasing, $w_j$ is nonnegative and nondecreasing in $t$. Set $w_j(0^-):=0$, and let $\nu_{j,0}$ be the Lebesgue--Stieltjes measure $dw_j$ restricted to $[0,\lambda)$. Thus $\nu_{j,0}([0,t])=w_j(t)$ for $0\le t<\lambda$, which is a finite measure as $w_j(\lambda^-)<\infty$. 

For every $a\ge0$ and $0\le t<\lambda$, we have
\[
(P_{j,a}-P_j)([0,t])=\bP(X_j>t)\indc{a\le t}.
\]
It follows that
\begin{align*}
\int_{[0,\lambda)}(P_{j,a}-P_j)([0,t])\,\nu_{j,0}(da)=\bP(X_j>t)w_j(t)=\bE\qth{(X_j-\mu_j)\indc{X_j>t}}=\int_{[0,t]}(\mu_j-x)\,P_j(dx).
\end{align*}
Since $g_j(x)$ vanishes for $x\ge \lambda$, Fubini's theorem yields
\begin{align*}
\cov(S,X_j\mid S<\lambda)&=\frac{1}{p}\int_{[0,\lambda)}\left(\int g_j(x)\,(P_{j,a}-P_j)(dx)\right)\nu_{j,0}(da)\\
&=\int_{[0,\lambda)}(E-E_{j,a})\frac{p_{j,a}}{p}\,\nu_{j,0}(da).
\end{align*}
All integrals are absolutely convergent because $g_j$ is bounded and $\nu_{j,0}$ is finite. Defining $\nu_j(da):=(p_{j,a}/p)\nu_{j,0}(da)$ gives a finite nonnegative measure on $\bR_+$, supported on $[0,\lambda)$, and proves the identity.
\end{proof}

\begin{proof}[Proof of \Cref{lem:safe}]
Since $S=\sum_{j=1}^nX_j$, linearity of conditional covariance and \Cref{lem:covariance} give
\begin{align*}
0<\var(S\mid S<\lambda)=\sum_{j=1}^n\cov(S,X_j\mid S<\lambda)=\sum_{j=1}^n\int_{[0,\infty)}(E-E_{j,a})\,\nu_j(da).
\end{align*}
At least one of these integrals is positive. Since $\nu_j$ is nonnegative, there is an $a\ge0$ for this index $j$ such that $E-E_{j,a}>0$. The cap move replacing $X_j$ by $\min(X_j,a)$ is therefore strictly safe.
\end{proof}

\subsection{The terminal state}\label{subsec:terminal_state}
By \Cref{lem:safe}, if no strictly safe move can be made, it must hold that $\var(S\mid S<\lambda)=0$; we call it a \emph{terminal state}. The next lemma establishes a lower bound on $\bP(S<\lambda)$ at a terminal state. 

\begin{lemma}\label{lem:zero-variance}
Suppose that $\sum_{i=1}^n\mu_i<\lambda$. If $\var(S\mid S<\lambda)=0$, then $E=\sum_{i=1}^n \alpha_i$, and
\begin{align*}
\bP(S < \lambda) \ge \prod_{i=1}^n \pth{1-\frac{\mu_i-\alpha_i}{\lambda-E}}, 
\end{align*}
with $\alpha_i := \operatorname{ess\,inf} X_i$. The equality holds if and only if $X_i \in \{\alpha_i, \alpha_i+\lambda-E\}$ almost surely for all $i\in [n]$. 
\end{lemma}

\begin{proof} 
By independence, $\operatorname{ess\,inf}S=\sum_{i=1}^n\alpha_i$. The assumption $\var(S\mid S<\lambda)=0$ implies that $S=E<\lambda$ almost surely on $\{S<\lambda\}$, thus $E=\operatorname{ess\,inf}S=\sum_{i=1}^n\alpha_i$. Because the random variables $X_i$ are independent, for each $1\le i\le n$, we have
\[
X_i=\alpha_i\quad\mbox{or}\quad X_i\ge \lambda - \sum_{j\neq i}\alpha_j = \alpha_i+\lambda -E
\]
almost surely. By Markov's inequality, for each $1\le i \le n$, 
\begin{align*}
    \bP(X_i=\alpha_i) = 1-\bP\left(X_i\ge \alpha_i+ \lambda -E\right) \ge 1- \frac{\bE[X_i-\alpha_i]}{\lambda -E} = 1- \frac{\mu_i-\alpha_i}{\lambda -E}.
\end{align*}
Since $0\le\mu_i-\alpha_i\le\sum_{j=1}^n\mu_j-E<\lambda-E$, every factor on the right is strictly positive. As a result,
\begin{align*}
\bP(S<\lambda) = \prod_{i=1}^n \bP(X_i = \alpha_i) \ge \prod_{i=1}^n \pth{1-\frac{\mu_i-\alpha_i}{\lambda-E}}. 
\end{align*}
Since all these factors are strictly positive, equality holds in the product if and only if equality holds in each application of Markov's inequality, which gives the stated characterization.
\end{proof}

\subsection{Proof of Theorem \ref{thm:refined-samuels}}\label{subsec:proof-samuels}
Now we have all the necessary ingredients to prove \Cref{thm:refined-samuels}. We iteratively construct random variables $X^{(t)}$, for each $t\ge 0$. Initially, we set $X_i^{(0)}=X_i$, $\mu_i^{(0)}=\mu_i$, and $E^{(0)}=E$. For each $t\ge 0$, if $\var(S^{(t)}\mid S^{(t)}<\lambda) > 0$, \Cref{lem:safe} shows that there is a strictly safe cap move from $X^{(t)}$; we choose such a move to ensure that
\[
E^{(t)}-E^{(t+1)}\ge\frac12\sup_{j\in[n],\,a\ge0}\pth{E^{(t)}-E_{j,a}^{(t)}}>0,
\]
where $E_{j,a}^{(t)}$ denotes the conditional expectation after capping $X_j^{(t)}$ at $a$. If $\var(S^{(t)}\mid S^{(t)}<\lambda) = 0$, we terminate the process. 

We first verify that all conditioning events $\{S^{(t)} < \lambda\}$  arising in the recursion have positive probability. Since $S^{(t)}$ is nonincreasing in $t$, it suffices to verify $\bP(S^{(0)}<\lambda) > 0$. Suppose the contrary, then $S^{(0)}\ge \lambda$ almost surely, a contradiction to $\bE[S^{(0)}] < \lambda$. As a result, the prerequisites $\bP(S<\lambda) > 0$ and $\sum_{i=1}^n \mu_i < \lambda$ of \Cref{lem:one-move,lem:safe,lem:zero-variance} hold throughout the recursion. 

Let $T\in\{0,1,\dots\}\cup\{\infty\}$ be the stopping time. If $T<\infty$, let $(X^{(T)},\mu^{(T)},E^{(T)},G^{(T)})$ denote the respective quantities at termination. If $T=\infty$ (i.e., the process continues indefinitely), we define $X_i^{(\infty)}=\lim_{t\to\infty}X_i^{(t)}$ for every $i=1,\dots,n$. Since each $X_i^{(t)}$ remains nonnegative and does not increase under a cap move, the decreasing limit $X_i^{(\infty)}$ exists almost surely. Similarly, the quantities $(\mu^{(\infty)},E^{(\infty)},G^{(\infty)})$ are well-defined as the decreasing limits of $(\mu^{(t)},E^{(t)},G^{(t)})$. Indeed, $\mu^{(t)}$ is nonincreasing because $X^{(t)}$ is nonincreasing, $E^{(t)}$ is nonincreasing by the definition of a safe move in \Cref{defn:safe}, and $G^{(t)}$ is nonincreasing by \Cref{lem:one-move}. Dominated convergence yields
\begin{align*}
\mu_i^{(\infty)} &:= \bE[X_i^{(\infty)}] = \lim_{t\to\infty} \bE[X_i^{(t)}] = \lim_{t\to\infty} \mu_i^{(t)}, \\
E^{(\infty)} &:= \frac{\bE[ S^{(\infty)}{\bf 1}_{S^{(\infty)}<\lambda}]}{\bP(S^{(\infty)}<\lambda)} = \frac{\lim_{t\to\infty}\bE[ S^{(t)}{\bf 1}_{S^{(t)}<\lambda}]}{\lim_{t\to\infty}\bP(S^{(t)}<\lambda)} = \lim_{t\to\infty} E^{(t)}, \\
G^{(\infty)} &:= \sum_{i=1}^n \mu_i^{(\infty)} - E^{(\infty)} = \lim_{t\to\infty} \pth{ \sum_{i=1}^n \mu_i^{(t)} - E^{(t)} } = \lim_{t\to\infty} G^{(t)}. 
\end{align*}
Consequently, $(X^{(T)}, \mu^{(T)}, E^{(T)}, G^{(T)})$ is well-defined in all cases.

Next we show that at termination, $\var(S^{(T)}\mid S^{(T)}<\lambda) = 0$. This is clearly true if $T<\infty$. Suppose now that $\var(S^{(\infty)}\mid S^{(\infty)}<\lambda)>0$. By \Cref{lem:safe}, there are $j\in[n]$ and $a\ge0$ such that $E_{j,a}^{(\infty)}<E^{(\infty)}$. For this fixed cap, write $S_{j,a}^{(t)}:=\sum_{i\ne j}X_i^{(t)}+\min(X_j^{(t)},a)$, for $t\in\{0,1,\dots\}\cup\{\infty\}$. Then $S_{j,a}^{(t)}\downarrow S_{j,a}^{(\infty)}$ almost surely, and  dominated convergence gives
\[
E_{j,a}^{(t)}=\frac{\bE[S_{j,a}^{(t)}\indc{S_{j,a}^{(t)}<\lambda}]}{\bP(S_{j,a}^{(t)}<\lambda)}\longrightarrow E_{j,a}^{(\infty)}<E^{(\infty)}.
\]
Set $\varepsilon:=(E^{(\infty)}-E_{j,a}^{(\infty)})/2>0$. Since $E^{(t)}\to E^{(\infty)}$, for all sufficiently large $t$ we have $E^{(t)}-E_{j,a}^{(t)}\ge\varepsilon$, so the choice of the move at step $t$ implies
\[
E^{(t)}-E^{(t+1)}\ge\frac12\pth{E^{(t)}-E_{j,a}^{(t)}}\ge\frac{\varepsilon}{2}>0.
\]
This contradicts the convergence of $E^{(t)}$. Therefore $\var(S^{(T)}\mid S^{(T)}<\lambda) = 0$ in all cases.

With all the above preparations, \Cref{lem:one-move} and \Cref{lem:zero-variance} imply the following inequality. 

\begin{lemma}\label{lem:telescoping}
For each $i=1,\dots,n$, let
\begin{align*}
z_i = \sum_{0\le t<T} (G^{(t)}-G^{(t+1)}){\bf 1}_{i_t = i} + (\mu_i^{(T)} - \operatorname{ess\,inf} X_i^{(T)}), 
\end{align*}
where $i_t$ is the coordinate moved at time $t$. Then $0\le z_i\le \mu_i^{(0)}$ and $\sum_{i=1}^n z_i = G^{(0)}$. In addition, 
\begin{align}\label{eq:telescoping}
\bP(S^{(0)}<\lambda) \ge \prod_{i=1}^n \pth{1-\frac{z_i}{\lambda-E^{(0)}}}. 
\end{align}
If equality holds, then the equality condition of \Cref{lem:zero-variance} holds for $X^{(T)}$, and either $T=0$, or $T>0$ and $G^{(1)}=0$.
\end{lemma}
\begin{proof}
By \Cref{lem:one-move}, $G^{(t)}\ge G^{(t+1)}$, so $z_i\ge 0$. Since the move at time $t$ changes only coordinate $i_t$ and is strictly safe, we also have
\[
G^{(t)}- G^{(t+1)}=\mu_{i_t}^{(t)}-\mu_{i_t}^{(t+1)}+E^{(t+1)}-E^{(t)}\le\mu_{i_t}^{(t)}-\mu_{i_t}^{(t+1)}.
\]
Therefore,
\begin{align*}
z_i\le\sum_{\substack{0\le t<T\\i_t=i}}\pth{\mu_i^{(t)}-\mu_i^{(t+1)}}+\mu_i^{(T)}- \operatorname{ess\,inf} X_i^{(T)} =\mu_i^{(0)}-\operatorname{ess\,inf} X_i^{(T)} \le\mu_i^{(0)}. 
\end{align*}
Also, at termination, \Cref{lem:zero-variance} gives $\sum_{i=1}^n(\mu_i^{(T)} - \operatorname{ess\,inf} X_i^{(T)}) = \sum_{i=1}^n \mu_i^{(T)} - E^{(T)} = G^{(T)}$, so 
\[
\sum_{i=1}^n z_i=\sum_{0\le t<T}\pth{G^{(t)}-G^{(t+1)}}+\sum_{i=1}^n \pth{\mu_i^{(T)} - \operatorname{ess\,inf} X_i^{(T)}} =G^{(0)}-G^{(T)}+G^{(T)}=G^{(0)}.
\]

For the main inequality \eqref{eq:telescoping}, applying \Cref{lem:one-move} successively gives
\begin{align*}
\bP(S^{(0)}<\lambda) \ge \prod_{0\le t<T} \pth{1-\frac{G^{(t)}-G^{(t+1)}}{\lambda-E^{(t)}}} \bP(S^{(T)}<\lambda). 
\end{align*}
This inequality also holds for $T=\infty$: $S^{(t)}\downarrow S^{(T)}$ implies $\bP(S^{(t)}<\lambda)\uparrow \bP(S^{(T)}<\lambda)$, so we may pass to the limit in the finite-step inequalities. Since $E^{(t)}\le E^{(0)}$ by definition of safe moves, and $\var(S^{(T)}\mid S^{(T)}<\lambda) = 0$ at termination, we apply \Cref{lem:zero-variance} to obtain
\begin{align*}
\bP(S^{(0)}<\lambda) \ge \prod_{0\le t<T} \pth{1-\frac{G^{(t)}-G^{(t+1)}}{\lambda-E^{(0)}}} \prod_{i=1}^n \pth{1-\frac{\mu_i^{(T)} - \operatorname{ess\,inf} X_i^{(T)}}{\lambda-E^{(0)}}}. 
\end{align*}
For each coordinate $i$, the elementary product inequality $\prod_{k=1}^m(1-u_k)\ge 1-\sum_{k=1}^m u_k$ for nonnegative $u_k$ with $\sum_{k=1}^m u_k<1$, followed by a limit if necessary, yields
\begin{align*}
\pth{1-\frac{\mu_i^{(T)} - \operatorname{ess\,inf} X_i^{(T)}}{\lambda-E^{(0)}}} \prod_{\substack{0\le t<T \\ i_t = i}}\pth{1-\frac{G^{(t)}-G^{(t+1)}}{\lambda-E^{(0)}}} \ge 1 - \frac{z_i}{\lambda - E^{(0)}}. 
\end{align*}
Combining the above displays gives \eqref{eq:telescoping}. 

For the equality case, clearly the equality condition of \Cref{lem:zero-variance} must hold at termination, i.e., for $X^{(T)}$. For the other statement, suppose $T>0$. The first move is strictly safe, so $E^{(1)} < E^{(0)}$, and $E^{(t)}<E^{(0)}$ for every $t\ge 1$. The same argument above, retaining the denominator after the first move, gives
\begin{align*}
\bP(S^{(0)}<\lambda) \ge \pth{1-\frac{G^{(0)}-G^{(1)}}{\lambda-E^{(0)}}} \prod_{1\le t<T} \pth{1-\frac{G^{(t)}-G^{(t+1)}}{\lambda-E^{(1)}}} \prod_{i=1}^n \pth{1-\frac{\mu_i^{(T)} - \operatorname{ess\,inf} X_i^{(T)}}{\lambda-E^{(1)}}}. 
\end{align*}
If $G^{(1)}>0$, then $\sum_{1\le t<T}(G^{(t)}-G^{(t+1)})+\sum_{i=1}^n (\mu_i^{(T)}-\operatorname{ess\,inf} X_i^{(T)})=G^{(1)}>0$, so at least one numerator in the last two products is positive. Replacing $E^{(1)}$ by $E^{(0)}$ therefore strictly decreases the right-hand side, because all the remaining factors have positive product. Thus equality with $T>0$ forces $G^{(1)}=0$, as claimed.
\end{proof}

By \Cref{lem:telescoping}, it suffices to lower bound the right-hand side of \eqref{eq:telescoping} subject to the constraints $0\le z_i\le \mu_i^{(0)}$ and $\sum_{i=1}^n z_i = G^{(0)}$. This optimization program is carried out in the following lemma. 
\begin{lemma}\label{lem:greedy}
Suppose that $z_1,z_2,\ldots ,z_n$ are nonnegative real numbers such that $z_i\le \mu_i$ for each $1\le i\le n$ and $$\sum_{i=1}^nz_i=\sum_{i=1}^n{\mu_i}-E,$$
with $0\le\mu_1\le \dots\le \mu_n$ and $0\le E\le \sum_{i=1}^n \mu_i<\lambda$. Then
$$\prod_{i=1}^n\left(1-\frac{z_i}{\lambda -E}\right)\geq \left( 1- \frac{ \sum_{i=1}^{K}\mu_i- E }{ \lambda - E } \right) \prod_{i=K+1}^n \left( 1- \frac{\mu_i}{ \lambda - E } \right),$$
where $K$ is the smallest index such that $E\le\sum_{i=1}^K\mu_{i}$. Equality holds if and only if
$$\{z_1,\ldots,z_n\}=\left\{0,\ldots,0,\sum_{i=1}^K\mu_{i}-E,\mu_{K+1},\ldots,\mu_{n}\right\}$$
as multisets. 
\end{lemma}
\begin{proof}
By our assumption, the vector
$$\left(0,\ldots,0,\sum_{i=1}^K\mu_{i}-E,\mu_{K+1},\ldots,\mu_{n}\right)$$
majorizes the vector $(z_1,\ldots,z_n)$. All coordinates lie in $[0,\sum_{i=1}^n\mu_i-E]\subseteq [0,\lambda-E)$, where $x\mapsto\log(1-x/(\lambda-E))$ is strictly concave. Karamata's inequality proves the bound, with equality if and only if the two vectors are permutations of one another.
\end{proof}

Since $G^{(0)} = \sum_{i=1}^n \mu_i^{(0)} - E^{(0)}$ and our initialization gives $(\mu_1^{(0)},\dots,\mu_n^{(0)},E^{(0)}) = (\mu_1,\dots,\mu_n,E)$, \Cref{thm:refined-samuels} follows from \Cref{lem:telescoping,lem:greedy}.

\subsection{Proof of Theorem \ref{thm:equality}}
Let $p:=\bP(S<\lambda)$, and $B(\mu_1,\dots,\mu_n,E,\delta)$ denote the right-hand side of \Cref{thm:refined-samuels}. We use the identity
\begin{equation}\label{eq:identity}
(\lambda-E)p=\bE[(\lambda-S)_+]=\lambda-\bE[S]+\bE[(S-\lambda)_+]=\delta+\bE[(S-\lambda)_+],
\end{equation}
where $x_+:=\max(x,0)$. In particular, $p=\delta/(\lambda-E)$ if and only if $S\le\lambda$ almost surely. In addition, since $\prod_{k=1}^m(1-u_k)\ge 1-\sum_{k=1}^m u_k$ for nonnegative $u_k$ with $\sum_{k=1}^m u_k<1$, we also have
\begin{align*}
B(\mu_1,\dots,\mu_n,E,\delta) = \left( 1- \frac{ \sum_{i=1}^{K}\mu_i- E }{ \lambda - E } \right) \prod_{i=K+1}^n \left( 1- \frac{\mu_i}{ \lambda - E } \right) \ge 1 - \frac{\sum_{i=1}^n\mu_i - E}{\lambda - E} = \frac{\delta}{\lambda- E}. 
\end{align*}

Suppose first that condition~(a) holds. The preceding displays and \Cref{thm:refined-samuels} give
\[
p=\frac{\delta}{\lambda-E}\le B(\mu_1,\dots,\mu_n,E,\delta)\le p,
\]
so equality holds in \Cref{thm:refined-samuels}.

Suppose next that condition~(b) holds. Summing its multiset identity gives $\sum_{i=1}^n(\mu_i-\alpha_i)=G$, hence $\sum_{i=1}^n\alpha_i=E$. Since $X_i\in\{\alpha_i,\alpha_i+\lambda-E\}$ almost surely, the event $S<\lambda$ occurs exactly when $X_i=\alpha_i$ for every $i$: their lower values sum to $E<\lambda$, and any upper value raises the sum to at least $\lambda$. Independence and the mean constraints therefore give
\[
p=\prod_{i=1}^n\bP(X_i=\alpha_i)=\prod_{i=1}^n\pth{1-\frac{\mu_i-\alpha_i}{\lambda-E}}=B(\mu_1,\dots,\mu_n,E,\delta),
\]
where the last equality follows from the multiset identity in condition~(b).

Conversely, suppose that equality holds in \Cref{thm:refined-samuels}, so $p=B(\mu_1,\dots,\mu_n,E,\delta)$. Apply the recursion from \Cref{subsec:proof-samuels}, and let $T$ and $z_i$ be as in \Cref{lem:telescoping}. By \Cref{lem:telescoping,lem:greedy},
\[
p\ge\prod_{i=1}^n\pth{1-\frac{z_i}{\lambda-E}}\ge B(\mu_1,\dots,\mu_n,E,\delta)=p.
\]
Thus equality holds in both lemmas. If $T=0$, set $\alpha_i:=\operatorname{ess\,inf}X_i$. The equality statement of \Cref{lem:telescoping} and the characterization in \Cref{lem:zero-variance} give $X_i\in\{\alpha_i,\alpha_i+\lambda-E\}$ almost surely. Moreover, $z_i=\mu_i-\alpha_i$, so equality in \Cref{lem:greedy} gives precisely the multiset identity in condition~(b).

If $T>0$, equality in \Cref{lem:telescoping} forces $G^{(1)}=0$. Since the gaps are nonnegative and nonincreasing by \Cref{lem:one-move}, this implies $G^{(t)}=0$ for all $t\ge 1$. At termination, the nonnegative quantities $\mu_i^{(T)}-\operatorname{ess\,inf}X_i^{(T)}$ sum to $G^{(T)}=0$, so they all vanish as well. Consequently, only the first move contributes to the definition of $z_i$, i.e., $z_i=G\indc{i=i_0}$, where $i_0$ is the first coordinate moved. Equality in \Cref{lem:telescoping} now yields
\[
p=\prod_{i=1}^n\pth{1-\frac{z_i}{\lambda-E}}=1-\frac{G}{\lambda-E}=\frac{\delta}{\lambda-E}.
\]
By \eqref{eq:identity}, we have $\bE[(S-\lambda)_+]=0$, hence $S\le\lambda$ almost surely. This is condition~(a), completing the proof.

\section{Theorem~\ref{thm:refined-samuels} implies Theorem~\ref{conj:samuels}}\label{sec:3}
In this section, we prove that the original Samuels' inequality (Theorem~\ref{conj:samuels}) is a corollary of Theorem~\ref{thm:refined-samuels}. 
\begin{proof}[Proof of Theorem~\ref{conj:samuels}]
    
As a function of $E$, the right-hand side $B(\mu_1, \ldots, \mu_n, E, \delta)$ of Theorem~\ref{thm:refined-samuels} is continuous on $\left[0, \sum_{i=1}^n \mu_i\right]$ and smooth on $\left(\sum_{i=1}^{K-1} \mu_{i}, \sum_{i=1}^{K} \mu_{i}\right)$ for each $1 \le K \le n$. 

Let $v := \frac{1}{\delta + \sum_{i=1}^n \mu_i - E}\in\left[\frac{1}{\delta + \sum_{i=1}^n \mu_i}, \frac{1}{\delta}\right]$. Then the right-hand side of Theorem~\ref{thm:refined-samuels} equals 
\[
f(v) := B\left(\mu_1, \ldots, \mu_n, \delta + \sum_{i=1}^n \mu_i - \frac{1}{v}, \delta\right).
\]
Since $\log f$ is concave in $v$ on each $\left(\frac{1}{\delta + {\sum_{i=K}^n} \mu_i}, \frac{1}{\delta + {\sum_{i=K+1}^n} \mu_i}\right)$ ($1 \le K \le n$), we have 
\[
f(v) \ge \min_{1 \le K \le n+1} f\left(\frac{1}{\delta + {\sum_{i=K}^n} \mu_i}\right) = \min_{1 \le i \le n} \prod_{j=i}^n \pth{1 - \frac{\mu_j}{\delta + \sum_{k=i}^n \mu_k}}.
\]
Therefore, the conclusion follows.
\end{proof}

\bibliographystyle{alpha}
\bibliography{refs}

\end{document}